\documentclass[smallextended]{svjour3}

\smartqed 

\usepackage{graphicx}
\usepackage{amsmath} 
\usepackage{amsfonts}

\newcommand{\R}{\mathbf{R}}
\newtheorem{assumption}{Assumption}
\newcommand{\reff}[1]{{\rm(\ref{#1})}}

\title{Some non monotone schemes for time dependent Hamilton-Jacobi-Bellman equations in stochastic control}
\titlerunning{Schemes for Hamilton-Jacobi-Bellman equations}
\author{Xavier Warin }
\institute{X. Warin \at EDF R\&D \& FiME, Laboratoire de Finance des March\'es de l'Energie (www.fime-lab.org)\\
           Tel: +33-1-47654184\\
           \email{xavier.warin@edf.fr}
           }

\begin{document}
\maketitle
\begin{abstract}
We introduce some approximation schemes for linear and fully non-linear diffusion equations of Bellman type.  Based on modified high order interpolators, the schemes proposed are not monotone but 
one can prove their convergence to the viscosity solution of the problem.  
Some of these schemes are related to a scheme previously proposed without proof of convergence.
Effective implementation of these schemes  in a parallel framework is discussed. They are  extensively 
tested on some simple test case, and on some  difficult ones where theoritical results of convergence  are not available. 
\end{abstract}

\section{Introduction}
\label{intro}
We are interested in a classical stochastic control problem whose value function is solution of the following Hamilton Jacobi equations:
\begin{eqnarray}
\frac{\partial v}{\partial t}(t,x) &  -&  \inf_{a_t \in \mathop{A}} \left( \frac{1}{2} tr(\sigma_{a_t}(t,x)\sigma_{a_t}(t,x)^T D^2 v(t,x)) + b_{a_t}(t,x) D v(t,x) \right.  \nonumber  \\
  &  & \left .  + c_{a_t}(t,x) v(t,x)+ f_{a_t}(t,x) \vphantom{\int_t} \right)  =  0  \mbox{ in  } \mathop{Q} \nonumber  \\
v(0,x)& = & g(x) \mbox{ in } \R^d
\label{hjb}
\end{eqnarray}
where $Q := (0,T] \times \R^d$,  $\mathop{A}$ is a complete metric space.
$\sigma_{a_t}(t,x)$  is a $d \times q$ matrix   and so  $\sigma_{a_t}(t,x) \sigma_{a_t}(t,x)^T$ is a $d \times d$ symmetric matrix, the  $b_{a_t}$ and $f_{a_t}$ coefficients are functions defined on  $Q$ with values ​​respectively in $\R^d$ and $R$.\\
Let's introduce an $\R^d$-valued controled process $X^{x,t}_s$ defined on a filtered probability space $\left(\Omega,\mathcal{F},\mathbb{F},\mathbb{P}\right)$ by
\begin{eqnarray}
dX^{x,t}_s & = & b_{a_s}(t,X^{x,t}_s) ds + \sigma_{a_s}(s,X^{x,t}_s) dW_s \nonumber \\
X^{x,t}_t &=& x \nonumber
\end{eqnarray}
where  $a =(a_s)$ is  a progressive process with values in  $\mathop{A}$.
This kind of problem arise when you want to minimize a cost function  $$J(t,x,a) =  \mathbb{E}[\int_{t}^T f_{a_s}(s,X^{x,t}_s) e^{\int_t^s c_{a_s}(u,X^{x,t}_u) du} ds + e^{\int_t^T c_{a_s}(u,X^{x,t}_u) du}  g(X^{x,t}_T)]$$ with respect to the control $a$. It is well known \cite{soner} that the optimal value $ \hat J(t,x) = \inf_{a} J(t,x,a)$ is a viscosity solution of equation \reff{hjb}.\\
Several approaches exist to solve this problem:
\begin{itemize}
\item The first approach is the generalized finite differences method developed by Bonnans Zidani \cite{bonnans1} where the derivatives are approximated  taking some  non  directly adjacent points. Directions are chosen such that the operator is consistent and it is monotone. Barles Souganidis framework \cite{barlessouganidis}  can be used to prove that the scheme converges to the viscosity solution of the problem.
\item The second is the semi-Lagrangian approach developed by Camilli Falcone for example  \cite{Camilli}, generalized by Munos Zidani \cite{Zidani} and studied in detail by Debrabant Jakobsen  \cite{Jakobsen}. In this approach, the scheme is discretized in time with a step $h$, the brownian motion is discretized taking two values ​​of the order of  $\sqrt{h}$.  The scheme still follows the Barles Souganidis framework.
\item The third approach is based on Monte Carlo techniques and the resolution of a Second Order Backward Stochastic Differential Equation. Initially
developed by Fahim, Touzi, Warin \cite{warin} for two particular schemes, it has been generalized for degenerated HJB equations by  Tan \cite{xiaolu}. The convergence  of the scheme to the viscosity solution is still given by the Barles Souganidis framework.
\end{itemize}
We will look at the use of semi-Lagrangian schemes for solving the control problem \reff{hjb}.
These schemes have been studied in detail recently by Debrabant and Jakobsen \cite{Jakobsen} but for a low degree interpolator (typically linear) that gives a monotone operator.
Under an assumption of CFL type, they show that the schemes are converging to the viscosity solution
of the problem and using the method of shaking coefficients \cite{Krilov}, they provide an estimate of the rate of convergence. Finally, they develop a higher order scheme but without proof of convergence to the viscosity solution of the problem. \\
In this note, we will look at non monotone approximations of higher degree.
We show  the convergence of such schemes to the viscosity solution of the problem and give an estimate of the error based on the fineness of the mesh.\\
The structure of the paper is as follows:
In the first part, we give the notations and some classical results of existence and uniqueness of the solution of \reff{hjb}.
After time discretization, we show for some general approximations the convergence of the discrete solution to the viscosity solution of the problem if we can solve the optimization problem obtained at each time step. \ \
We explain why the monotony is not necessary to obtain convergence toward the viscosity solution: a scheme converging under certain assumptions always converges to the viscosity solution.\\
In the second part we develop several Lagrange interpolators, spline interpolators, and approximations based on Bernstein polynomials converging. \\
The last part, the techniques to properly treat the boundary conditions and parallelization methods are developed so that problems of dimension greater than 2 can be tackled. On different test cases, we show that the developed schemes are effective even in cases where the theory gives no evidence of convergence (unbounded control).

\section{Notation and regularity results}
\label{notation}
We denote by $\wedge$ the minimum and $ \vee$ the maximum.
We denote by $|\quad | $ the Euclidean norm of a vector. 
For a bounded function $w$, we set
\begin{eqnarray}
 | w|_0  = \sup_{(t,x) \in Q} | w(t,x)|, &  \quad  & [w]_1 = \sup_{(s,x) \ne (t,y)} \frac{ |w(s,x) -w(t,y)|}{|x-y| + |t-s|^{\frac{1}{2}}} \nonumber
\end{eqnarray}
and  $|w|_1 = | w|_0 +[w]_1 $. $C_1(Q)$ will stand for the space of functions  with a finite  $|\quad |_1$ norm.\\
For  $t$ given, we denote
\begin{eqnarray}
||w(t,.) ||_{\infty} = \sup_{x \in \R^d } | w(t,x)| \nonumber
\end{eqnarray}
We use the classical assumption on the data of \reff{hjb} for a given $\hat K$: 
\begin{eqnarray}
 \sup_a |g|_1 + |\sigma_a|_1 + |b_a|_1 +  |f_a|_1  + | c_a|_1 \le \hat K
\label{coeff}
\end{eqnarray}

The following proposition \cite{Jakobsen} gives us the existence of a solution in the space of bounded Lipschitz functions
\begin{proposition}
If the coefficients of the equation  \reff{hjb} satisfy \reff{coeff}, there exists a unique viscosity solution of the equation \reff{hjb}  belonging to  $C_1(Q)$.  If $ u_1 $ and $u_2$ are  respectively sub and supersolution of equation  \reff{hjb} satisfying  $u_1(0,.) \le u_2(0,.)$  then  $u_1 \le u_2$.
\end{proposition} 
A spatial discretization length of the problem $\Delta x$  being given, thereafter  $(i_1 \Delta x ,..,i_d \Delta x)$ with $\bar i = (i_1 , ..., i_d) \in \mathbf{ Z}^d$  will correspond to the coordinates of a mesh   $M_{\bar i}$ defining a hyper-cube in dimension $d$.
For an interpolation grid $(\xi_{i})_{i=0,..N} \in [-1,1]^N$, and for a mesh $\bar i$, the point $y_{\bar i, \tilde j}$ with $\tilde j = (j_1, .., j_d) \in [0,N]^d$  will have the coordinate
$(\Delta x (i_1 + 0.5 (1+\xi_{j_1})) ,..,\Delta x (i_d + 0.5 (1+\xi_{j_d}))$.
We denote  $(y_{\bar i, \tilde j})_{\bar i, \tilde j}$ the set of all the grids points on the whole domain.\\
We notice that for regular mesh with constant volume  $\Delta x^d$, we have the following relation for all $x \in \R^d$:
\begin{eqnarray}
\min_{\bar i, \tilde j} | x- y_{\bar i, \tilde j}| \le \Delta x
\label{MaxDist}
\end{eqnarray}
Finally, from one line to the other constants $C$ may be changed.
 
\section{General discretization}
\label{generalDis}
The equation \reff{hjb} is discretized in time by the scheme proposed by Camilli Falcone \cite{Camilli}  for a time discretization $h$.
\begin{eqnarray}
v_h(t+h,x) & = &  \inf_{a \in \mathop{A}}  \left[ \sum_{i=1}^q \frac{1}{2q} (v_h(t, \phi^{+}_{a, h, i}(t,x)) +  v_h(t, \phi^{-}_{a, h, i}(t,x)) )
  \right. \nonumber \\ 
&  & \left. +  f_a(t,x) h + c_a(t,x) h v_h(t,x) \vphantom{\int_t}  \right]  \nonumber \\
    & := & v_h(t,x)+ \inf_{a \in \mathop{A}}   L_{a,h}(v_h)(t,x)   \label{hjbCamilli}  
\end{eqnarray}

with
\begin{eqnarray}
L_{a,h}(v_h)(t,x) & = & \sum_{i=1}^q \frac{1}{2q} (v_h(t, \phi^{+}_{a, h, i}(t,x)) + v_h(t, \phi^{-}_{a, h, i}(t,x)) -  2 v_h(t,x))  \nonumber \\
              &    &  + h c_a(t,x) v_h(t,x)+ h f_a(t,x) \nonumber \\
\phi^{+}_{a, h, i}(t,x) & =& x +b_a(t,x) h + (\sigma_a)_i(t,x) \sqrt{h q}  \nonumber\\
\phi^{-}_{a, h, i}(t,x) & =& x +b_a(t,x) h - (\sigma_a)_i(t,x) \sqrt{h q}  \nonumber
\end{eqnarray}
where  $(\sigma_a)_i$ is the $i$-th column of $\sigma_a$. We note that it is also possible to choose other types of discretization in the same style as those defined
in \cite{Zidani}.\\
In order to define the solution  at each date,  a condition on the value chosen for $v_h$ between $0$ and $h$ is required. We choose a time linear interpolation once the solution has been calculated at date  $h$:
\begin{eqnarray}
v_h(t,x) = (1- \frac{t}{h})  g(x)+  \frac{t}{h} v_h(h,x), \forall t \in [0,h]. 
\label{hjbCamilli1}
\end{eqnarray}
We first recall the following result :
\begin{proposition}
\label{converDisH}
  Under the condition on the coefficients  given by equation \reff{coeff}, the solution $v_h$ of equations \reff{hjbCamilli} and \reff{hjbCamilli1} is uniquely defined and belongs to   $C_1(Q)$. We check that if $ h \le (16 \sup_a \left \{ |\sigma_a|_1^2 + | b_a|_1^2 +1 \right \} \wedge 2 \sup_a |c_a|_0)^{-1} $, there exists $C$ such that 
\begin{eqnarray}
| v -v_h |_0 \le C h^{\frac{1}{4}} 
\end{eqnarray}
Moreover, there exists  $C$ independent of $h$ such that 
\begin{eqnarray}
\label{LipschitVh}
| v_h|_0 & \le & C \\
|v_h(t,x)-v_h(t,y)| & \le & C |x-y|, \forall (x,y) \in Q^2
\end{eqnarray}
\end{proposition}
\begin{proof}
The existence of a solution in $C_1(Q)$ is an application of the  proposition 8.4 in \cite{Jakobsen}. 
The error estimate corresponds to the theorem 7.2  in \cite{Jakobsen}.
The existence of a   Lipschitz constant uniform in  $x$  and  the uniform bound are given by the corollary  8.3 in the same article.
\end{proof}
It is assumed throughout this section for simplicity that  $N$ is fixed. 
For a function $v$ from  $\R^d$ to $\R$, we denote  the  set of all the values taken by $v$  on the grids  $(y_{\bar i, \tilde j})_{\bar i, \tilde j}$ by
$(v^{\bar i, \tilde j})_{\bar i, \tilde j}$.
We define the operator  $T_{\rho}$ ($\rho = (h,\Delta x)$) with values in $C(\R^d)$ such that   $T_{\rho}((v^{\bar i, \tilde j})_{\bar i, \tilde j})$ is an approximation of $v$ (not necessarily an interpolation). In the sequel we still note  $T_{\rho}$
the operator defined on the set of function $v$ from $\R^d$ to $\R$ by $T_{\rho} v :=  T_{\rho}((v^{\bar i, \tilde j})_{\bar i, \tilde j}).$
\\
We consider the HJB equation discretized at the grids points:
\begin{eqnarray}
v^{\bar i,\tilde j}_{\rho}(t+h) & =&  v^{\bar i,\tilde j}_{\rho}(t)   \nonumber \\
& & + \inf_{a \in \mathop{A}}  \left [  ( L_{a,h} T_{\rho} ( (v^{\bar k,\tilde l}_{\rho}(t))_{\bar k, \tilde l }))(t,y_{\bar i, \tilde j})  \right]  \label{hjbInterpGeneral} 
\end{eqnarray}
with a linear interpolation between  $0$ and  $h$ following \reff{hjbCamilli1}.\\
We denote  $\tilde v_{\rho}(t) = T_{\rho}(( v^{\bar i, \tilde j}_{\rho}(t))_{(\bar i, \tilde j)})$ the reconstructed solution in  $C(R^d)$.\\
We will now describe the approximation operator so that the scheme  converges to the viscosity solution.
We will extend the notion of weight developed by \cite{Jakobsen}.
\begin{assumption}
\label{HypoPoids}
Suppose that   $T_{\rho}$ is an operator from $C_0(Q)$ to  $C_0(Q)$, that there exists a function of $h$  $ \tilde K_h \xrightarrow{h \longrightarrow 0} 0$ such that 
for  $x \in M_{\bar i}$
\begin{eqnarray}
(T_{\rho} f)(x) & = & \sum_{\tilde j \in [0,N]^d } (w^h_{\bar i, \tilde j}(f))(x) f(y_{\bar i, \tilde j}) \\
0 \le (1 - \tilde K_h h) & \le &  \sum_{\tilde j \in [0,N]^d} (w^h_{\bar i, \tilde j}(f))(x)  \le 1+ \tilde K_h h 
\end{eqnarray} 
and that the functions  $w^h_{\bar i, \tilde j}(f)$ are positive weights functions depending on  $f$, $h$ and the support $M_{\bar i}$.
\end{assumption}
\begin{remark}
The previous operator is more general than the one defined by  \cite{Jakobsen}:
\begin{itemize}
\item A priori it depends on  $h$. It allows us to  accept some reconstructed solutions even if small oscillations are present. 
\item We don't impose  $w^h_{\bar i, \tilde j}(f)(y_{\bar k, \tilde l}) = \delta_{\bar i \bar k} \delta_{\tilde j,\tilde l}$ such that the approximation operator is not necessarily an interpolation operator.
\item The weight depends on the fonction $f$ which is used.
\end{itemize}
\end{remark}
The following theorem shows that any approximation operator satisfying the above assumptions converges to the viscosity solution.
\begin{theorem}
Suppose $T_{\rho}$ satisfies the assumptions  \ref{HypoPoids}.
We consider a sequence  $ \rho_p=(h_p,\Delta x_p) \longrightarrow (0,0)$  such that  $\frac{\Delta x_p}{h_p}  \longrightarrow  0$,  $h_p \le  1$.
Let's build a solution  $\tilde v_{\rho_p}$ of \reff{hjbInterpGeneral} for all  $p$,
 then $\tilde v_{\rho_p}$ converges to the viscosity solution of \reff{hjb}.
Moreover for $h_p$ small enough there exists $C$ independent on  $h_p$, $N$, $\Delta x_p$ such that 
\begin{eqnarray}
|\tilde v_{\rho} - v |_0 & \le &  C ( h_p^{\frac{1}{4}}+\frac{\Delta x_p}{h_p} + \tilde K_{h_p}) 
\end{eqnarray}
\label{TheoConv}
\end{theorem}
\begin{proof}
Choose $h \le   1$ and satisfying the hypothesis of proposition \reff{converDisH}.
We directly estimate  $\tilde v_{\rho} - v_h$ . Introduce
\begin{eqnarray}
e(t) & =& ||\tilde v_{\rho}(t,.) - v_h(t,.)||_{\infty} \nonumber
\end{eqnarray}
By definition of   $T_{\rho}$, for a given point  $x$ in $M_{\bar i}$
\begin{eqnarray}
|\tilde v_{\rho}(t,x) - v_h(t,x) |  & \le & |\sum_{\tilde j \in [0,N]^d } w^h_{\bar i, \tilde j}(\tilde v_{\rho})(x)  (v_{\rho}^{\bar i, \tilde j}(t) - v_h(t,x))|   \nonumber\\
& &  +  |\sum_{\tilde j \in [0,N]^d } w^h_{\bar i, \tilde j}(\tilde v_{\rho})(x)  -1| \quad  |v_h(t,x)| \nonumber \\
                             &  \le &  \sum_{ \tilde j \in [0,N]^d } w^h_{\bar i, \tilde j }(\tilde v_{\rho})(x)) |v_{\rho}^{\bar i, \tilde j}(t) - v_h(t,x)|  \nonumber \\ 
                             &  & + \tilde K_h h |v_h|_0 \nonumber \\
                            & \le & (1+ \tilde K_h h) |v_{\rho}^{\bar i, \tilde k} - v_h(t,x)|  \nonumber \\ 
                             &  & + \tilde K_h h |v_h|_0
\label{eqnJac}
\end{eqnarray}
with  $y_{\bar i, \tilde k}$ such that  $|v_{\rho}^{\bar i, \tilde k}(t) - v_h(t,x)|$ maximizes $|v_{\rho}^{\bar i, \tilde j}(t) - v_h(t,x)| $.\\
Moreover we denote 
\begin{eqnarray}
(\hat L_{a,h} v)(t,x)& =& \frac{1}{2q} \sum_{i=1}^q (v(t,\phi^{+}_{a,h,i}(t,x ))+  v(t,\phi^{-}_{a,h,i}(t,x)))   \nonumber \\
&  &  + h c_a(t,x) v(t,x) + h f_a(t,x)
\end{eqnarray}
such that $V :=  v_{\rho}^{\bar i, \tilde k}(t) - v_h(t,x) $ satisfies
\begin{eqnarray}
V  &  = &   \inf_a [ (\hat L_{a,h} \tilde v_{\rho})(t-h,y_{\bar i, \tilde k})  - \inf_a   (\hat L_{a,h} v_{h})(t-h,x) ] \nonumber 
\end{eqnarray}
So using $| inf . - inf . | \le sup | . - .|$ we get 
\begin{eqnarray}
|V|  & \le &     \frac{1}{2q}  \sum_{i=1}^q   \left [\sup_a | \tilde v_{\rho}(t-h,\phi^{+}_{a,h,i}(t-h,y_{\bar i, \tilde k}))  \right.  \nonumber \\
  & & - v_{h}(t-h,\phi^{+}_{a,h,i}(t-h,x))| \nonumber\\
  &     &  \left . + \sup_a | \tilde v_{\rho}(t-h, \phi^{-}_{a,h,i}(t-h,y_{\bar i, \tilde k}))-    v_{h}(t-h,\phi^{-}_{a,h,i}(t-h,x)) |  \right ]\nonumber \\
  &   & +   h \sup_a |c_a|_0 |v_{\rho}^{\bar i, \tilde k}(t-h) - v_h(t-h,y_{\bar i, \tilde k})| + h |v_h|_0 |c_a|_1 \Delta x |    \nonumber \\
  & &  +    h \sup_a| f_a|_1 \Delta x \nonumber 
\end{eqnarray}
Using the fact that the data in \reff{hjb} belong to  $C_1(Q)$, such that 
\begin{eqnarray}
\label{DiffPhi}
| \phi^{-}_{a,h,i}(t,y_{\bar i, \tilde k}) -\phi^{-}_{a,h,i}(t,x) | & \le & \Delta x(1 +  \sup_a |b_a|_1 h  + \sup_a |(\sigma_a)_i|_1 \sqrt{h q})  \nonumber\\
                                                        & \le & \Delta x (1+C (\sqrt{h q} + h))  \nonumber \\
                                                        & \le & C \Delta x 
\end{eqnarray}
and using the fact that $|v_h|_1$ is bounded independently on  $h$, one gets the estimate
\begin{eqnarray}
| \tilde v_{\rho}(t-h,\phi^{+}_{a,h,i}(t-h,y_{\bar i, \tilde k})) & -& v_{h}(t-h,\phi^{+}_{a,h,i}(t-h,x))| \le \nonumber \\
& &  || \tilde v_{\rho}(t-h,.) - v_h(t-h,.)||_\infty + \nonumber \\
& & C |v_h|_1 \Delta x \nonumber 
\end{eqnarray}
Using the fact that $|f_a|_1$, $|c_a|_1$  are bounded independently of  $a$:
\begin{eqnarray}
|V| & \le &  || \tilde v_{\rho}(t-h,.) - v_h(t-h,.)||_\infty +   \nonumber \\
              & & +  h \sup_a |c_a|_0 |v_{\rho}^{\bar i, \tilde k}(t-h) - v_h(t-h,y_{\bar i, \tilde k})|  \nonumber \\ 
         & &  +  C |v_h|_1 \Delta x  +  h |c_a|_1 |v_h|_0 \Delta x +   h \sup_a| f_a|_1 \Delta x \nonumber  \\
              & \le & || \tilde v_{\rho}(t-h,.) - v_h(t-h,.)||_\infty (1+ h \hat K ) + \Delta x C \nonumber
\end{eqnarray}
where the constant  $C$ depends on  $\tilde K$, $|v_h|_1$, $\hat K$.\\
So
\begin{eqnarray}
 |v_{\rho}^{\bar i, \tilde k}(t) - v_h(t,x)| & \le & e(t-h) (1+ h \hat K) + C \Delta x 
\end{eqnarray}
By combining the above equation with \reff{eqnJac}:
\begin{eqnarray}
e(t) & \le &  (1+ \tilde K_h h)(1+ h \hat K)   e(t-h) +  C (\Delta x + h \tilde K_h) \nonumber 
\end{eqnarray}
so there exists $\hat C$ such that
\begin{eqnarray}
e(t) & \le &  (1+ \hat C h)  e(t-h) +  C (\Delta x + h \tilde K_h) \nonumber 
\end{eqnarray}
Moreover applying the previous iteration at the first time step :
\begin{eqnarray}
e(h)  & \le &  (1+ \hat C  h) |g|_0 + C (\Delta x + h \tilde K_h) \nonumber 
\end{eqnarray}
and by using the definition of  $v_{\rho}$ on $[0,h]$ given by \reff{hjbCamilli1}
\begin{eqnarray}
e(t)& \le & \frac{t}{h} \left[(1+ \tilde K_h h) |g|_0 + C (\Delta x + h \tilde K_h) \right] + (1- \frac{t}{h}) |g|_0, \forall t \le h \nonumber 
\end{eqnarray}
Using the discrete Gronwall lemma 
\begin{eqnarray}
e(t) & \le &  C (\frac{\Delta x}{h} + \tilde K_h)  e^{ \hat C T}    \forall t \le T \nonumber 
\end{eqnarray}
Moreover by using  $|\tilde v_{\rho_p} - v|_0 \le | \tilde v_{\rho_p} - v_{h_p}|_0 + |v_{h_p} -v|_0$ and the proposition \reff{converDisH} we get the final result.
\end{proof}
\begin{remark}
The suppositions  on the  weight function assure that the scheme is a nearly monotone one. The approximation of the function leads to a global scheme
which is the perturbation of a  monotone one and it is not surprising that ist is converging towards the viscosity solution (see remark 2.1 in \cite{barlessouganidis})
\end{remark}
\begin{remark}
 Under assumption \ref{HypoPoids} with $\tilde K_h=0$ we just impose that on each cell $M_{\hat i}$, given the values $(f(y_{\hat i, \tilde j})_{\tilde j})$,  all reconstructed values $(T_{\rho} f)(x)$ 
 are in  between the $\min_{\tilde j}(f(y_{\hat i, \tilde j}))$ and $\max_{\tilde j}(f(y_{\hat i, \tilde j}))$.
Taking  $\tilde K_h$ not null and decreasing to $0$ permits to release the previous condition and to get reconstructed values slightly below  $\min_{\tilde j}(f(y_{\hat i, \tilde j}))$ 
or above $\max_{\tilde j}(f(y_{\hat i, \tilde j}))$. An application of the result above  permits to get convergence results for some new schemes in the section below and for some schemes in the litterature where authors
could not get any convergence results.
\end{remark}
\begin{remark}
Using a high order scheme  on each mesh $M_{\tilde i}$ won't improve the theoretical rate of convergence, but improving the consistency at least locally we hope that the observe rate of convergence will be higher.
\end{remark}

\section{Some approximating operators}
\label{approx}
In this section we first develop some methods based on Lagrange  interpolators  and splines.
We examine in particular the Lagrange interpolators using the Gauss Lobatto Legendre and Gauss Lobatto  Chebyshev interpolators associated to some truncation. 
We will also consider the case of cubic splines and monotone cubic splines used by \cite{Jakobsen} which have the characteristic of not requiring truncation.
In the last section, we will detail some  polynomial approximation of Bernstein type that also do not require truncation and provide a  monotone scheme.

\subsection{Truncated Lagrange interpolators}
For more information on the Lagrange interpolators and their properties, one can refer to Appendix \reff{Annexe1}.
In this section, we suppose that a Lagrange interpolator with grid points  $X =(\xi_i)_{i=0,N} \in [-1,1]^{N+1}$ is given.
The space is discretized with meshes  $M_{\bar i} =  \displaystyle \prod_k [x_{i_k}, x_{i_k}+ \Delta x ]$  with $\bar i = (i_1,..,i_d)$ and on each mesh a Lagrange interpolator
 $I^{X}_{\Delta x, N }$ is defined by tensorization giving a multidimensional interpolator with   $(N+1)^d$ points.
We are particularly interested in Gauss Lobatto Chebyshev and  Gauss Lobatto Legendre interpolators that have a low Lebesgue constant  and thus avoid oscillations.
On a mesh $M_{\tilde i}$ and for a point  $x$ in this mesh, we note  $\underline v_{\bar i} = \displaystyle  \min_{\tilde j} v(y_{\bar i, \tilde j})$, $\bar v_{\bar i} = \displaystyle  \max_{\tilde j} v(y_{\bar i, \tilde j})$.
We introduce the following truncated operator:
\begin{eqnarray}
\hat I^{X}_{h, \tilde K_h, \Delta x, N }(v) & =&  (\underline v_{\bar i} -\tilde K_h h |\underline v_{\bar i}|)  \vee I^{X}_{\Delta x, N }(v) \wedge  (\bar v_{\bar i} + \tilde K_h  h |\bar v_{\bar i}|)  \nonumber 
\end{eqnarray}
where $\tilde K_h h < 1$ and $\tilde K_h \xrightarrow{h \longrightarrow 0} 0$.
\begin{proposition}
\label{InterpTruncProp}
The interpolator $\hat I^X_{h, \tilde K_h, \Delta x, N}$  has the following properties:
\begin{eqnarray}
|| \hat I^{X}_{h, \tilde K_h, \Delta x, N}(f)(x) ||_\infty & \le & (1 + \tilde K_h h) ||f||_\infty  \nonumber
\end{eqnarray}
There exists  $C_{N,d}$ such that for each Lipschitz bounded function $f$:
\begin{eqnarray}
|| \hat I^{X}_{h, \tilde K_h, \Delta x, N}(f)(x) -f(x) ||_\infty & \le &  (C_{N,d} \Delta x  K + \tilde K_h h )   |f|_1  
\label{InterpolTronc}
\end{eqnarray}
\end{proposition}
\begin{proof}
The first assertion is obtained by definition of the truncation.
The second assertion can be deduced from \reff{interpolX}:
if there is no truncation  in $x$
\begin{eqnarray}
| \hat I^{X}_{h, \tilde K_h, \Delta x, N}(f)(x) - f(x) | & \le &  C K \Delta x \frac{(1+\lambda_N(X))^d}{N+2}\nonumber 
\end{eqnarray}
where $\lambda_N$ is the Lebesgue constant assiocated to  interpolator.
If there is truncation in $x$, for example a truncation to the maximum value, we  note $y_{\bar k, \tilde l }$ the point where  $f(y_{\bar i, \tilde j })$ is maximum and we suppose for instance that $f(y_{\bar i, \tilde j }) \ge 0$.
We have the relation:
\begin{eqnarray}
| \hat I^{X}_{h, \tilde K_h, \Delta x, N}(f) (x ) -f(x)| & = & | (1+\tilde K_h h) f(y_{\bar i, \tilde j }) -f(x)| \nonumber \\
                        & \le & \tilde K_h h |f|_0 + |f(y_{\bar i, \tilde j }) -f(x)| \nonumber \\
                        & \le & (\tilde K_h h  + K \Delta x) |f|_1 \nonumber
\end{eqnarray}
Of course the same result can be obtained with a minimum truncation.
\end{proof}
\begin{proposition}
The interpolator  $\hat  I^{X}_{h, \tilde K_h, \Delta x, N}$  satisfies the assumptions  \reff{HypoPoids} so that $ \tilde v_{h, \Delta x, N}$ converges to the viscosity solution.
Moreover 
\begin{eqnarray}
|| v - \tilde v_{\rho} ||_\infty \le    O( h^{\frac{1}{4}}) + O( \frac{ \Delta x}{ h}) + O(\tilde K_h) \nonumber
\end{eqnarray} 
\end{proposition}
\begin{proof}
Because of the truncation for each point  $x$ of a mesh $M_{\bar i}$, we have 
\begin{eqnarray}
(\hat  I^{X}_{h, \tilde K_h, \Delta x, N} (f)(x)  =  \underline w^h_{\bar i}(f)(x) \underline v_{\bar i} + \bar w^h_{\bar i}(f)(x) \bar  v_{\bar i} \nonumber
\end{eqnarray}
If $ \underline v_{\bar i}  \le  \hat  I^{X}_{h, \tilde K_h, \Delta x, N} (f)(x) \le  \bar  v_{\bar i}$ then
\begin{eqnarray}
 \underline w^h_{\bar i}(f)(x) & = & \frac{\hat  I^{X}_{h, \tilde K_h, \Delta x, N}(f)(x) - \bar  v_{\bar i}}{ \underline v_{\bar i} - \bar  v_{\bar i}}  \nonumber \\
\bar w^h_{\bar i}(f)(x) & = & 1 -\underline w^h_{\bar i}(f)(x)  \nonumber\\
0 & \le & \underline w^h_{\bar i}(f)(x)   \le  1  \nonumber
\end{eqnarray}
If $ \underline v_{\bar i} > \hat  I^{X}_{h, \tilde K_h, \Delta x, N} (f)(x)$
\begin{eqnarray}
 \bar w^h_{\bar i}(f)(x) &=& 0  \nonumber \\
\underline w^h_{\bar i}(f)(x)& =&  \frac{ \hat  I^{X}_{h, \tilde K_h, \Delta x, N} (f)(x)}{\underline v_{\bar i}}  \in [1-\tilde K_h h, 1+ \tilde K_h h]  \nonumber 
\end{eqnarray}
Otherwise
\begin{eqnarray}
 \underline w^h_{\bar i}(f)(x) &=& 0 \nonumber  \\
\bar w^h_{\bar i}(f)(x)& =& \frac{   \hat  I^{X}_{h, \tilde K_h, \Delta x, N} (f)(x)}{\bar v_{\bar i}}  \in [1-\tilde K_h h, 1+ \tilde K_h h]  \nonumber
\end{eqnarray}
Then choose the weight functions above associated to the points with values associated to the extremal points $\underline v_{\bar i}$, $\bar v_{\bar i}$
and take a weight equal to $0$ for other points.
The final estimation is obtained by proposition  \reff{converDisH} and theorem \reff{TheoConv}.
\end{proof}
\begin{remark}
Due to this estimation the truncation should be such that  $\tilde K_h = h^{\frac{1}{4}}$.
\end{remark}
We also give the consistency error:
\begin{proposition}
The consistency error is  in  $O( h  + \frac{\Delta x^2}{h})$ in areas where the truncation is achieved and  in $O( h  + \frac{\Delta x^{N+1}}{h})$ otherwise.
\end{proposition}
\begin{proof}
Let $u$ be the solution of  \reff{hjb} that we suppose regular.
Defining
\begin{eqnarray}
E(u)& = & \frac{1}{h}  |\sup_a [ u(t,x) -  (L_{a,h} (\hat I^X_{h,\tilde K_h, \Delta x, N} u))(t-h,x)] | 
\label{consistency}
\end{eqnarray}
we get 
\begin{eqnarray}
E(u) & \le &  \frac{1}{h}  \sup_a |u(t,x) -   (L_{a,h} u)(t-h,x)| \nonumber \\
& &   +    \sup_a | (L_{a,h} u)(t-h,x)-(L_{a,h}\hat I^X_{h,\tilde K_h, \Delta x, N} u)(t-h,x)|  \nonumber \\
& \le & \frac{1}{h} \sup_a |u(t,x) -   ( L_{a,h} u)(t-h,x)| \nonumber \\
&  & +  
 \frac{C}{h}|\hat I^{X}_{h,\tilde K_h, \Delta x, N} u -u|_0 \nonumber  
\end{eqnarray}
using the assumption \ref{coeff}.
Besides using consistency of the scheme with equation \reff{hjb} and assumption \ref{coeff}:
\begin{eqnarray*}
\frac{1}{h} |u(t,x) -   ( L_{a,h} u)(t-h,x)| & \le & C h (| \frac{\partial^2 u}{\partial t^2}|_0 + | \frac{\partial^2 u}{\partial x^2}|_0  + | \frac{\partial^3 u}{\partial x^3}|_0 +
| \frac{\partial^4 u}{\partial x^4}|_0)
\end{eqnarray*}
The interpolation error is given by \reff{InterpRegulier} when no truncation is achieved.
When the truncation is effective for a point $x \in M_{\bar i}$, it means for example that the non truncated interpolator  gives a value which is above all the values at the grid point of the mesh. 
So $I^{X}_{h,\tilde K_h,\Delta x,2} u(x) \le  \bar u_{\bar i}  = \hat I^{X}_{h,\tilde K_h, \Delta x, N} u(x) \le   I^{X}_{h,\tilde K_h, \Delta x, N} u(x)$
and  the approximation has an error in between $O(\Delta x ^2)$ and $O(\Delta x^{N+1})$. Then the consistency error with this term is at least the one obtained by the
linear interpolator.
\end{proof}
\subsection{Cubic spline interpolators}
A cubic spline interpolation is used to interpolate a one-dimensional function. No discretization point inside the mesh is given.
With $d=1$, keeping the same notations  as before, the grid points are $y_{\bar i, \tilde j}$ with  $\bar i$ the mesh number and $\tilde j =0$ or $1$ corresponding to the left or right part of the mesh. In particular $y_{\bar i, 1} = y_{\bar i+1, 0}$.
\subsection{Truncated cubic spline}
Let $I^{X_c}_{\Delta x, 1}$ be the cubic spline interpolator.
We use the truncated interpolator:
\begin{eqnarray}
\hat I^{X_c}_{h, \tilde k_h, \Delta x }(v) & =&  (\underline v_{\bar i} -\tilde K_h h |\underline v_{\bar i}|)  \vee I^{X_c}_{\Delta x, 1 }(v) \wedge  (\bar v_{\bar i} + \tilde K_h  h |\bar v_{\bar i}|) \nonumber 
\end{eqnarray}
It is clear that the interpolator satisfies the assumptions \reff{HypoPoids} and that we satisfy the assumptions of theorem  \reff{TheoConv}.\\
As before, the consistency order depends on the fact the truncation has been performed or not.
\begin{proposition}
The consistency error with the interpolator    $\hat I^{X_c}_{h, \tilde k_h, \Delta x }$ is  in  $O( h  + \frac{\Delta x^2}{h})$  in areas where the truncation is achieved and  in $O( h  + \frac{\Delta x^{4}}{h})$ otherwise.
\end{proposition}
\begin{proposition}
The solution $\tilde v_{\rho}$ obtained by the interpolator   $\hat I^{X_c}_{h, \tilde k_h, \Delta x }$ converges to the viscosity solution $v$ and the convergence rate is given by:
\begin{eqnarray}
|| v - \tilde v_{\rho} ||_\infty \le    O( h^{\frac{1}{4}}) + O( \frac{ \Delta x}{ h}) + O(\tilde K_h) \nonumber
\end{eqnarray}
\end{proposition}
\subsection{Monotone cubic spline  (\cite{Jakobsen})}
It is possible to modify the cubic spline algorithm to obtain a monotone interpolation by direction (but not globally monotone) so that the interpolated function is $C_1$. It is achieved by modifying the estimated derivatives used by the spline following the Eisenstat Jackson Lewis algorithm  \cite{Eisenstat} (derived from the Fritsch-Carlson algorithm). This ensures the monotony of the interpolated function.
Debrabant and Jakobsen have modified this algorithm by relaxing the continuity of the derivative so that the  interpolation is reduced to a local problem on the mesh and adjacent cells.
This interpolation is of order 4 in the mesh if the interpolated function is monotone.
By tensorization, using  Remark 5.1 in \cite{Jakobsen}, the  non-monotone interpolator operator in \cite{Jakobsen}  in dimension $d$  can be written:
\begin{eqnarray}
I^S_{\Delta x}(f)(x) & = &  \sum_{\bar i, \tilde j } w^h_{\bar i, \tilde j  }(f)(x) f(y_{\bar i, \tilde j}) \nonumber \\
\mbox{where } & &  \mbox{ the support of  } w^h_{\bar i, \tilde j}(f) \mbox{ is } M_{\bar i} \nonumber \\
   & &  w^h_{\bar i, \tilde j}(f) \ge 0 , \nonumber \\
   & &  \sum_{\bar i, \tilde j} w^h_{\bar i, \tilde j}(f)(x) =1 
\label{InterpSpline}
\end{eqnarray} 
It is clear that the interpolator satisfies assumption \reff{HypoPoids} and the assumptions of theorem \reff{TheoConv}. As shown in \cite{Jakobsen}
\begin{proposition}
If the interpolated function $u(t-h,.)$ is monotone between grid points, the consistency error with the interpolator  $I^S_{\Delta x}$  given by equation \ref{consistency} is   $O( h  + \frac{\Delta x^{4}}{h})$ 
\end{proposition}
As a direct result of theorem \reff{TheoConv} we get the convergence of the scheme that was not given in 
\cite{Jakobsen} :
\begin{proposition}
The solution  $ \tilde v_{h,\Delta x,1}$ obtained by interpolator  $I^{S}_{\Delta x }$ converges to the viscosity solution of \reff{hjb} and
\begin{eqnarray}
|| v - \tilde v_{\rho} ||_\infty \le    O( h^{\frac{1}{4}}) + O( \frac{ \Delta x}{ h}) \nonumber 
\end{eqnarray}
\end{proposition}
\begin{remark}
In fact it is shown in \cite{Kocic} that when the data is non monotone, the Fritsch-Carlson type algorithm (which has been modified to get the Eisenstat Jackson Lewis algorithm)
is only clipping the solution to the maximum of the interpolated points so is equivalent to a truncation with $\tilde K_h =0$, so in that case $\hat I^{S}_{\Delta x}=I^{S}_{\Delta x}$ . So the local consistency
error is similar to the one obtained by the other scheme developed when truncation is achieved.
\end{remark}
\subsection{Approximation with Bernstein polynomials}
The weights associated to Bernstein polynomials are positive (Appendix \reff{Annexe3}), independent on the function.  Their sum is equal to one and we get nearly all the assumptions
used by  \cite{Jakobsen} except the fact that this is not an interpolator.  By using the results given in appendix \reff{Annexe3} we deduce that
\begin{proposition}
The scheme with Bernstein approximation $B_N$ of degree $N$ (in each dimension)  converges  to the viscosity solution of \reff{hjb} with
\begin{eqnarray}
|| v - \tilde v_{\rho} ||_\infty \le    O( h^{\frac{1}{4}}) + O( \frac{ \Delta x}{ h}) \nonumber 
\end{eqnarray}
 and 
the consistency error is of order  $O(h + \frac{\Delta x^2}{N h })$.
\end{proposition}

\section{Some numerical results}
In this section we focus on techniques for effective implementation of Semi Lagrangian algorithms.
We are interested in any special treatment of the boundary conditions that can avoid problems with this kind of algorithm. 
The parallelization strategy is investigated and  on numerical examples, we calculate the rate of convergence of the different methods on
conventional tests from \cite{Jakobsen} and \cite{Zidani}.  We eventually use numerical tests with unbounded controls to show that
the methods work even outside the theoretical framework of convergence.\\
We insist that in our tests any meshes are taken: in particular, the discretizations  do not respect the monotony of functions and discretization parameters are not chosen so that the approximation points are inside the domain.
 If the value of a function must be estimated outside the domain, the scheme is amended as indicated in the following paragraph. If no change is possible we truncate the solution projected on the edge of the domain.
The order of the estimate may be lower than the theoretical one or the one given by \cite{Jakobsen} but closer to a real use of the schemes.

\subsection{Boundary conditions}
\label{boundsection}
The boundary conditions are often problematic for PDEs and their treatment by the Semi Lagrangian methods exacerbates the problem.
Indeed, if for example we solve a problem with $ b $ and $ \sigma $ constant for simplicity and if $ x $ is a mesh point near the edge then $ x + bh + \sqrt{h} \sigma $ can be out of the domain resolution. This problem occurs if a point is too close to the edge and if the volatility is too large or the time steps too small.
A first possibility which is quite natural is to interpolate the solution outside the domain or to set  it to a given value.
The interpolation is to be avoided  as much as possible because it causes oscillations that can explode during resolution.
The first trick is to modify the schema. You can often avoid fetching points outside the area by changing the points sought by the interpolation.
In Figure \reff{Boundfig} we show how a 1D scheme starting from a point $x$ may need a point out of the domain.
\begin{figure}[h]
\centering
\includegraphics[width=8cm]{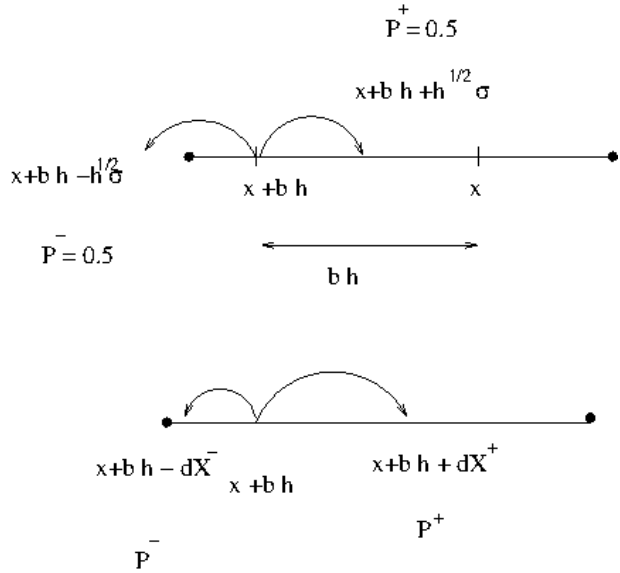}
\caption{Modification of the scheme for boundary conditions}
\label{Boundfig}
\end{figure}
The points used can be modified (respecting 'mean' and 'variance'). We denote $dX^{-}$ the difference between $x+b h$ and the point reached by below and $dX^{+}$
the distance with the point reached by above. 
If the scheme is not modified, 
$dX^{+} =dX^{-} = \sigma \sqrt{h}$ 
and the  'probability' to reach these  points are  
$P^{-}=P^{+} = \frac{1}{2}$.
If a value has to be interpolated outside the domain, new weights and new interpolation points inside the domain are calculated respecting
\begin{eqnarray}
dX^{+} dX^{-} & =&  \sigma^2 h \nonumber  \\
P^{+} &=&  \frac{\sigma^2 h}{(dX^{+})^2 +\sigma^2 h}   \nonumber \\
P^{-} &=&  \frac{(dX^{+})^2}{(dX^{+})^2 +\sigma^2 h}  \nonumber
\end{eqnarray}
In the general case this modification of the ``probabilities'' force us to modify the interpolation point in the other directions.
In the corners of the domain the modification thus can be impossible and some kind of extrapolation has to be used.
\begin{remark}
The Bonnans Zidani method has the same flaw: when the scheme needs some points outside the domain, the consistency or the monotonicity has to be relaxed.
\end{remark}
\begin{remark}
The use of this methods clearly doesn't satisfy   the assumption in \cite{Jakobsen} and their results should be adapted to get convergence results similar to the one obtained in proposition \ref{converDisH}.
The consitency error due to time discretization  (solving equation \ref{hjbCamilli}) cannot be better than $h^{\frac{1}{2}}$.
\end{remark}

\subsection{Parallelization technique}
In order to solve a stochastic control problems in high dimension (3 or above) parallelization techniques are required.
Of course thread parallelization can be easily added to these techniques.
Suppose that we have 4 processors and that the grid of points is split between processor (figure  \reff{Parallfig}). At the initial date, each processor
has its own data (the initial solution). At the first time step, each processor needs some data owned by other processors : some values
needed by the interpolation. The control being bounded one can determine the envelop of  the points needed by the processor.
On figure  \reff{Parallfig}, we give the data needed by processor 3 for its optimization.
\begin{figure}[h]
\centering
\includegraphics[width=8cm]{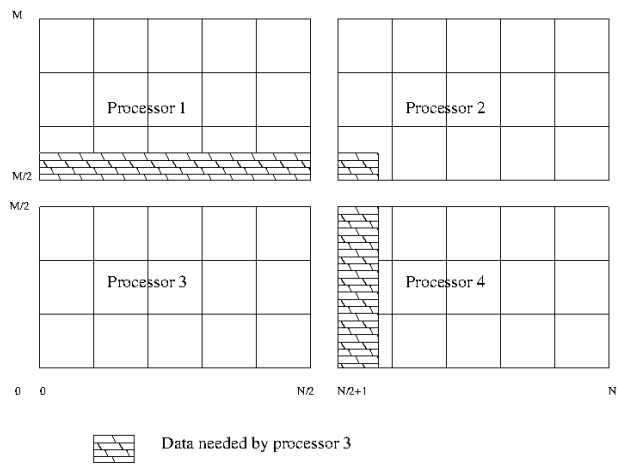}
\caption{Data to send to processor  3}
\label{Parallfig}
\end{figure}
Then some MPI communications are realized once at each time step.
This algorithm is very effective because communications are just achieved only once at each time step and negligible in time spent.
It has already been proved to be very effective till thousand of processors in dynamic programming problem in high dimension \cite{warinHPC1}, \cite{warinHPC2}.
Although it is specialy effective in high dimension permitting to tackle $4$ dimensional problems, it remains very attractive even on two dimension problems.
In figure below, we report acceleration obtained for test case \reff{mpiTestCase} below on figure \reff{mpifig} in the special case where we have a number of time steps is equal to $100$, a number of mesh equal to $200$ in each direction and a linear interpolator.
\begin{figure}[h]
\centering
\includegraphics[width=6cm]{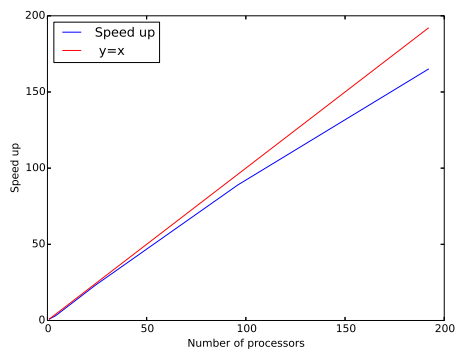}
\caption{Mpi acceleration for Semi-Lagragian Schemes for a 2D problem}
\label{mpifig}
\end{figure}

\subsection{Some test cases}
In this section we give some results for the explicit scheme with polynomial approximation of Bernstein type (BERN i where i
is the degree of the polynomial), with linear interpolation (LIN), with Chebyshev interpolation (TCHEB i), with Legendre interpolation (LEGEND i), 
with cubic splines (CUBIC), with monotone cubic splines (MPCSL following the name given by \cite{Jakobsen}).
The first 3  are taken from the literature, but by extending the domain of resolution to get  highly non-monotone solutions.
The two last do not fit into the framework of the theory because the condition given by equation \ref{coeff} is not verified. They are
however interesting because the methods developed are effective.
Unlike \cite{Jakobsen}, we  chose to set the same time step, and set the same control discretization for all the methods  and all the space discretization in a given test case.
By converging in space very thinly we should get some residual errors due to these fixed discretizations.
All truncation are achieved with  $K_h=0$. In the tables, NbM correspond to the number of mesh per direction, Err the error with respect to the analytical or reference solution in infinite norm,
Time correspond to the CPU times for the resolution, and Rate to the order of convergence numerically calculated: if $Err(n)$ is the error obtained with  $n$ meshes per direction,
the order of convergence with $4n$ meshes is given by $ \frac{log((Err(n)-Err(2n))/(Err(2n)-Err(4n))}{log2}$.
As for the boundary treatment, extrapolation outside the domain is used at points when  the methodology in section \ref{boundsection} is impossible to use.
In each test case, the function $g$ is obtained by taking the analytical solution with $t=0$.

\subsubsection{First test case without control \cite{Jakobsen}}
Coefficients are:
\begin{eqnarray}
f_a(t,x) & =& \sin x_1 \sin x_2 ((1+ 2 \beta^2)(2-t)-1)  \nonumber \\  
&  & -2 (2-t) \cos x_1 \cos x_2 \sin(x_1 + x_2) \cos(x_1+x_2)  \nonumber \\
c_a(t,x)  &= & 0,  \quad  b_a(t,x)= 0 \quad \sigma_a(t,x)  = \sqrt{2}  \left  ( \begin{array}{lll}
                                                                 \sin(x_1+x_2) & \beta & 0  \\
                                                                 \cos(x_1 +x_2) &  0 & \beta 
                                                                 \end{array}
                                                                 \right )  \nonumber 
\end{eqnarray}
We take  $\beta=0.1$ and solve the problem on   $Q = (0,1] \times [- 2 \pi, 2 \pi]^2$.
The analytical solution is  $u(t,x) = (2-t) \sin x_1 \sin x_2$. The number of time steps is equal to 2000 so $h=5e-4$. Considering the results in table \reff{Case1}, we can conclude that
on a regular linear problem:
\begin{itemize}
\item The order of convergence of the linear approximation is  below  $2$. Because the solution is smooth we could hope to get a rate of convergence equal to $2$ the consistency error for the LINEAR scheme : in fact we only get a rate equal to $1.3$ certain due to the boundary treatment.
 For  LEGEND of degree $2$,  the rate of convergence observed is  $3$ : so it is equal to the consistency rate observed certainly indicating that the truncation is not achieved. With the discretization tested
the boundary condition doesn't seem to perturb the solution.
For  CUBIC, MPCSL, TCHEB, LEGEND of degre 3  the rate of convergence is roughly $4$ but with some oscillations indicating that some truncations are achieved.
\item The cost of Chebyshev is twice the cost of the Legendre polynomials: a analysis shows that this is due to trigonometric functions that are costly in time.
\item The use of monotone spline is not superior to classical spline approximation with truncation.
\item Bernstein polynomial are not competive 
\item The three most effective schemes are the  CUBIC, MPCSL  and LEGEND with degree 2.
\end{itemize}
\begin{table}
\small
\caption{Test case 1  } 
\label{Case1}

\scalebox{0.6}{
 \begin{tabular}{|c|c|c|c|c|c|c|c|c|c|c|c|c|c|c|c|}
\hline
 \multicolumn{4}{c|}{LINEAR} & \multicolumn{4}{c|}{CUBIC} & \multicolumn{4}{c|}{MPCSL}  & \multicolumn{4}{c|}{TCHEB 3}   \\
\hline
 NbM &  Err & Rate  & Time  &  NbM &  Err & Rate  & Time  & NbM &  Err & Rate  & Time    & NbM &  Err & Rate  & Time     \\
\hline
240         &   0.310  &        &   112     &  20    & 0.461   &        &  2    & 20  &   0.815     &        &  3     &  20    &  0.165   &            &     31   \\ 
480         &   0.119  &        &   448     &  40    & 0.037   &         &  10   & 40  &   0.166     &        &  10   &  40    & 0.0086   &            &   133  \\ 
960         &   0.040  & 1.26   &  1815     &  80    & 0.005   & 2.84   &  41   & 80  &   0.007     & 4.62   &  42    &  80    & 0.00108  &    4.37    &  552    \\
1920        &  0.0075  & 1.29   &  7334     &  160   & 0.0005   & 4.52  &  165  & 160 &  0.0005     & 4.88   &  170  &  160    & 0.0003 &  3.40      & 2246\\
\hline
 \multicolumn{4}{c|}{LEGEND 2} & \multicolumn{4}{c|}{LEGEND 3} &   \multicolumn{4}{c|}{BERN 2} & \multicolumn{4}{c|}{BERN 3}\\
\hline
 NbM &  Err & Rate  & Time  &  NbM &  Err & Rate  & Time &  NbM &  Err & Rate  & Time  &  NbM &  Err & Rate  & Time  \\
\hline
 20       &  0.059    &           &  6    &  20      & 0.165   &    &  21   &    120       & 0.643   &             &   718     &  120  & 0.5528      &    &   2811 \\ 
 40       &  0.0069   &           &  25   &  40      & 0.0085   &   &  92    &   240       & 0.227   &         &  2878     &  240  & 0.1784     &   &  11292   \\ 
 80      &  0.0010   &   3.15    &  104  &  80      & 0.00107  &   4.38      &  380   &   480       & 0.077   &  1.47       &  11551    &  480  & 0.0557     &  1.60  &  45446    \\
160      &  0.0003   &   3.00    &  420  &  160     & 0.0003   &   3.41      &  1547  &     960       & 0.021   &  1.42       &  46467    &  960     & 0.01532           &  1.60    & 181897  \\
\hline

\end{tabular}

}
\end{table}

\subsubsection{A second test case without control \cite{Jakobsen}}
Its solution is not regular 
\begin{eqnarray}
u(t,x) &= & (1+t) \sin(\frac{x_2}{2}) \left \{ \begin{array}{ll}
                                              \sin \frac{x_1}{2}  &  \mbox{  for }  -2 \pi < x_1 < 0 \\
                                              \sin  \frac{x_1}{4}  &  \mbox{  for }  0 < x_1 < 2 \pi \\
                                              \end{array} \right.  \nonumber 
\end{eqnarray}
with 
\begin{eqnarray}
f_a(t,x) & =& \sin \frac{x_2}{2}  \left \{ \begin{array}{ll}
                     \sin \frac{ x_1}{2}  (1+ \frac{1+t}{4}) (\sin^2 x_1 + \sin^2 x_2)  &  \mbox{ for } -2 \pi < x_1 < 0 \\ 
                     \sin  \frac{ x_1}{4} (1+ \frac{1+t}{16}) (\sin^2 x_1 + 4 \sin^2 x_2)  &  \mbox{ for }  0   < x_1 < 2 \pi
                                          \end{array}
                                 \right .   \nonumber  \\
       &   & - \sin x_1 \sin x_2 \cos  \frac{x_2}{2}  \left \{ \begin{array}{ll}
                             \frac{1+t}{2} \cos \frac{ x_1}{2}  &  \mbox{ for } -2 \pi < x_1 < 0 \\ 
                             \frac{1+t}{4} \cos \frac{ x_1}{4}   &  \mbox{ for }  0   < x_1 < 2 \pi \\
                                           \end{array}
                                                   \right .  \nonumber  \\
c_a(t,x)  &= & 0,  \quad  b_a(t,x)= 0 \quad \sigma_a(t,x) =  \sqrt{2}  \left( \begin{array}{l}
                                                                             \sin x_1  \\
                                                                             \sin x_2  
                                                                 \end{array} 
                                                                 \right )  \nonumber 
\end{eqnarray}
On take  $Q = (0,1] \times [- 2 \pi, 2 \pi]^2$, the number of time step is equal to 2000 so $h=5e-4$.
Our previous results are confirmed and here Lagrange polynomial of degree two are the most effective.
Note that CUBIC and  MPCSL give the same results for these discretizations (it is not true for more coarse discretizations not given here). 
As expected the convergence rate dropped due to singularity to high order schemes. But even in this case the high order scheme remains far more effective.

\begin{table}
\caption{Test case 2   } 
\label{Case2}

\scalebox{0.6}{
 \begin{tabular}{|c|c|c|c|c|c|c|c|c|c|c|c|c|c|c|c|}
\hline
 \multicolumn{4}{c|}{LINEAR} & \multicolumn{4}{c|}{CUBIC} & \multicolumn{4}{c|}{MPCSL}&  \multicolumn{4}{c|}{TCHEB 3} \\
\hline
 NbM &  Err & Rate  & Time  &  NbM &  Err & Rate  & Time  & NbM &  Err & Rate  & Time & NbM &  Err & Rate  & Time \\
\hline
640         & 0.038   &         & 557    & 80     &  0.00875   &        & 17   &  80     & 0.00875   &          &  18  &  20       &    0.0136   &        & 12      \\ 
1280        & 0.013   &         & 2240   & 160    &  0.00439   &       & 70   &  160    & 0.00439   &           &  72  &   40  &    0.00398  &    &  51 \\ 
2560        & 0.0070  &   1.84 & 8662   & 320    &  0.00220   & 0.99  & 285    &  320    & 0.00220   &   0.99  & 288   &   80  &    0.00128       &  1.84  & 212  \\
5120        & 0.0035  &   0.99 & 34820  & 640    &  0.00110   & 0.99  & 1177 &  640    & 0.00110   &   1.00     & 1221   &  160 &    0.0005       &  1.97  & 858    \\
\hline
 \multicolumn{4}{c|}{LEGEND 2} & \multicolumn{4}{c|}{LEGEND 3} & \multicolumn{4}{c|}{BERN 2} & \multicolumn{4}{c|}{BERN 3} \\
\hline
 NbM &  Err & Rate  & Time  &  NbM &  Err & Rate  & Time  &  NbM &  Err & Rate  & Time  &  NbM &  Err & Rate  & Time  \\
\hline
 20       &    0.01422      &        &  2   &  20       &   0.0137     &    &    9      &    80       &  0.3774   &              &  112   &  80      & 0.3144  &       & 429         \\ 
 40       &    0.00411      &        & 11   &  40       &   0.0040     &    &   39       &    160       &  0.1889  &         &  448   &  160     &  0.135  &    & 1734   \\ 
  80       &    0.00132      &   1.85    & 46   &  80       &   0.00129 &  1.84  &  160  &  320       &  0.066   &   1.80      &  1794  &  320     &  0.0460 &  1.43 & 6937        \\ 
  160 &    0.0006           &    1.96 & 184 &     160       &  0.0006   &   1.96 &  649  &   640       &  0.0186  &   1.89       &  7197  &  640     &  0.0128 &  1.80 &  27892   \\
\hline
\end{tabular}
}
\end{table}
 In order to check that the singularity was slowing the convergence rate, the domain $D$ has been split into two parts. First  part $D_1$ (singularity area) is  for $x_1 \in [ -\frac{\pi}{8},  \frac{\pi}{8}]$, while the second is $D_2 = D \setminus D_{1}$. The error  and the convergence rate have been calculated for the two domains for LINEAR and CUBIC approximations in table \reff{Case2Sing}. As for the LINEAR scheme, the error remains mainly higher in the $D_2$ domain explaining why the global rate of convergence remains high. As for the CUBIC scheme, the error remains always lower in the $D_2$ domain and all the rate of convergence of the $D_1$ domain correspond to the rate of convergence of the global domain.
 \begin{table}
\caption{Test case 2 : error near and far away the singularity  } 
\label{Case2Sing}
 \begin{tabular}{|c|c|c|c|c|c|c|c|c|c|}
\hline
 \multicolumn{5}{|c|}{LINEAR} & \multicolumn{5}{|c|}{CUBIC} \\
\hline
    &  \multicolumn{2}{|c|}{$D_1$} &  \multicolumn{2}{|c|}{$D_2$}  &   & \multicolumn{2}{|c|}{$D_1$} &  \multicolumn{2}{|c|}{$D_2$} \\
\hline
 NbM & Err  & Rate   & Err & Rate & NbM & Err  & Rate  & Err & Rate\\
80   & 0.12  &       & 1.21 &           &  40    &  0.017  &       & 0.010   &        \\
160  & 0.085 &       & 0.36 &           &   80   &  0.0087 &       &  0.0033  &   \\
320  & 0.052 &  0.16 & 0.14 &  0.55     &  160   &  0.0043 &  0.96     &  0.00086   &  1.60   \\
640  & 0.027 &  0.46 & 0.038 &  1.15    &  320   &  0.0022 &  0.99     &  0.0005    &  2.73   \\
1280  & 0.013 &  0.85 & 0.010 &  1.83   &  640   &  0.0011 &  0.99     &  0.0005     &        \\
\hline
\end{tabular}
\end{table}

\subsubsection{Control problem with a regular solution \cite{Jakobsen}, \cite{Zidani}}
\label{mpiTestCase}
The regular solution is given by 
\begin{eqnarray}
 u(t,x_1,x_2 ) = (\frac{3}{2} -t ) \sin x_1 \sin x_2 \nonumber 
\end{eqnarray}
Coefficients are given by 
\begin{eqnarray}
f_a(t,x) & =& (\frac{1}{2}-t) \sin x_1 \sin x_2  +  (\frac{3}{2}-t) \left[ \sqrt{ \cos^2 x_1 \sin^2 x_2 + \sin^2 x_1 \cos^2 x_2} \right. \nonumber \\
         &  & \left . - 2 \sin(x_1+x_2) \cos(x_1+x_2) \cos x_1 \cos x_2 \right]  \nonumber  \\
c_a(t,x)  &= & 0,  \quad  b_a(t,x)=  a \quad \sigma_a(t,x) =  \sqrt{2}  \left( \begin{array}{l}
                                                                             \sin(x_1+x_2)  \\
                                                                             \cos(x_1+x_2)  
                                                                 \end{array} 
                                                                 \right ), \nonumber  \\
   \mathop{A}& = &  \{ a \in \R^2 : a_1^2 + a_2^2 =1 \} \nonumber  
\end{eqnarray}
$Q = (0,1] \times [-\pi, \pi]^2$ and the number of time steps is equal to 1000 so $h=1e-3$, the number of control equal to 4000.
CPU times are given for a number of core equal to 192. Once again the quadratic approximation Legend 2 is the most effective.
\begin{table}
\caption{Test case 3  } \label{Case3}
\scalebox{0.6}{
 \begin{tabular}{|c|c|c|c|c|c|c|c|c|c|c|c|c|c|c|c|c|c|c|c|}
\hline
 \multicolumn{4}{c|}{LINEAR} & \multicolumn{4}{c|}{CUBIC} & \multicolumn{4}{c|}{MPCSL} &  \multicolumn{4}{c|}{TCHEB 3} \\
\hline
 NbM &  Err & Rate  & Time  &  NbM &  Err & Rate  & Time  & NbM &  Err & Rate  & Time & NbM &  Err & Rate  & Time   \\
\hline
80  & 0.59  &          & 237     &    10   &  0.312   &       & 19  &  10     & 0.688   &        &  30  &   8  & 0.0986  &   &47  \\ 
160 & 0.147 &          & 850     &    20   &  0.0499  &      & 30  &  20     & 0.050   &         &  29  &   16 & 0.0119  &  & 184  \\ 
320 & 0.044 & 2.09    & 3334    &    40   &  0.0072  & 2.61  & 96  &  40     & 0.0064  & 3.86   &  98  &  32 & 0.0012 &  3.01  & 735    \\
640 & 0.014 & 1.77    & 13259   &    80   &  0.001   & 2.78  & 384 &  80     & 0.001   & 3.03   &  387 &   64 & 0.0008 & 4.94  & 2944   \\
\hline
 \multicolumn{4}{c|}{LEGEND 2} & \multicolumn{4}{c|}{LEGEND 3} & \multicolumn{4}{c|}{BERN 2} & \multicolumn{4}{c|}{BERN 3} \\
\hline
 NbM &  Err & Rate  & Time  &  NbM &  Err & Rate  & Time  & NbM &  Err & Rate  & Time  &  NbM &  Err & Rate  & Time  \\
\hline
 8       &  0.0710       &        & 14   & 8   &  0.0988     &       & 31       &    20       & 0.7479    &        &  181     &  20    & 0.769  &          & 758  \\ 
16      &  0.0094       &         & 49   & 16  &  0.0117     &    &  116     & 40       & 0.706     &       &  789     &  40    & 0.5898 &          & 2362 \\ 
32      &  0.0023       &  3.11   & 149   & 32  &  0.0011    &  3.04  &  465     &     80       & 0.3210    & 0.62  & 2533     &  80    & 0.2334 &  0.98  &  9436  \\ 
64      &  0.0009       &  2.36   & 590   & 64  & 0.0010     &   6.24 &  1854 &  160      & 0.0801    & 2.09  & 10111    &  160   & 0.0563 &  1.00  & 37750  \\
 \hline
\end{tabular}
}
\end{table}

\subsubsection{One dimensional optimization problem with unbounded control}
\label{stochasticTarget1D}
The theory is developed for bounded controls. One may wonder if we are able to solve problems with unbounded control.
We are interested in a stochastic target problem where  we want to drive a portfolio towards the value $1$ at $T$ with a given probability $x$.
The asset used for investment satisfies :
\begin{eqnarray*}
dS_t &= & \mu dt + \kappa dW_t
\end{eqnarray*}
Supposing a null interest rate, the wealth process of an investor investing in bond and the asset follows
\begin{eqnarray*}
dX^{\theta}_t & = & \theta_t  \mu dt + \theta_t \kappa dW_t
\end{eqnarray*}
where $\theta_t$ is the investor strategy.\\
Using the methodology developped in \cite{Bouchard}, the minimal value $u$  at date $t$  of the initial portfolio  
 to reach a target $1$ at date $T$ with probability $x$  satisfies by Ito lemma  for a smooth $u$, $s \ge t$ :
\begin{equation}
\left \{ 
\begin{array}{ccl}
du(s,p_s^{t,x,\alpha}) & = &  [ \frac{\partial u}{\partial t} + \frac{\alpha_s^2}{2} \frac{\partial^2 u}{\partial x^2}](s,p_s^{t,x,\alpha}) ds + \alpha_s \frac{\partial u}{\partial x}(s,p_s^{t,x,\alpha}) dW_s,\\
d P_s^{t,x,\alpha}    & = & \alpha_s dW_s \\
P^{t,x,\alpha}_t &= & x 
\end{array}  \right. \nonumber 
\end{equation}
The HJB equation is obtained by imposing that the variation of the wealth  is equal to the variation of $u(s,p_s^{t,x,\alpha})$ :
\begin{eqnarray*}
inf_{\alpha \frac{\partial u}{\partial x} = \kappa \theta} [ \frac{\partial u}{\partial t}(t,x) + \frac{\alpha^2}{2} \frac{\partial^2 u}{\partial x^2}(t,x) - \mu \theta] =0 \mbox{ for} (t,x) \in [0,T] \times [0,1]
\end{eqnarray*}
The final condition is  obviously given by $u(T,x)=x$.
So setting $T=1$, $\mu =0.1$,  $\kappa = 0.1$,    $u$ is the  solution  of equation \reff{hjb} with :
\begin{eqnarray}
f_a(t,x) & =& 0,  c_a(t,x)  =  - \mu \theta, \quad \sigma_a(t,x) = \alpha,  \mathop{A}(u) = \{ a=(\alpha, \theta) \in \R^2, \alpha u_x  = \kappa \theta  \}  \nonumber\\
u(0,x)  & =& x,  Q =  (0,1] \times [0,1] \nonumber
\end{eqnarray}
Using the first order condition, the function $u$ satifies
\begin{eqnarray*}
 \frac{\partial u}{\partial t}(t,x) - \frac{\mu^2}{2\kappa^2} \frac{(\frac{\partial u}{\partial x})^2}{\frac{\partial^2 u}{\partial x^2}} =0
\end{eqnarray*}
The Fenchel transform of $u$ , $v(t,q) = \sup_{x\in[0,1]}\{ xq -u(t,x)\}$ satisfies
\begin{equation}
\left \{ 
\begin{array}{l}
\frac{\partial v}{\partial t}(t,q) + \frac{\mu^2}{2\kappa^2} \frac{\partial^2 v}{\partial^2 q}(t,q) =0   \mbox{   for} (t,x) \in [0,T] \times \R \\
v(T,q) = (q-1)^{+}
\end{array}  \right. \nonumber 
\end{equation}
Using Feyman Kac, $v$ is the price of an European call where the asset follows the dynamic:
\begin{eqnarray*}
dQ_t = \frac{\mu}{\kappa} Q_t dW_t
\end{eqnarray*}
Using Black Scholes formulae and taking the dual of $v$ (so the bi-dual of $u$), we get the
analytical solution 
\begin{eqnarray}
u(t,x)& = & N ( N^{-1}(x) + \frac{\mu}{\kappa} \sqrt{T-t}) , \nonumber \\
N(x) & = & \frac{1}{\sqrt{2 \pi}} \int_{-\infty}^{x} e^{-\frac{u^2}{2}} du \nonumber
\end{eqnarray}
The set of controls depends on the solution. Numerically $a$ is bounded so that the diffusion coefficients don't explode.
We use the solution calculated at the previous time step  $u(t-h,x)$  to estimate  $ \mathop{A}(u(t,.)) \simeq \mathop{A}(u(t-h,.))$.\\
The solutions obtained for the various schemes are given in Table \reff{Case4}.
The controls are bounded to  $16$, the number of controls tested is equal to 8000 and the number of time steps is taken equal to 1600 so $h=6.25e-4$.
CPU times are given for $48$ cores used.
All the methods have similar convergence rate but for very coarse meshes high order schemes are far more effective. On the finer meshes
used, LEGEND 2, CUBIC and MPCSL still give the best results considering the error versus the computing time.
\begin{table}
\caption{Test case 4  } \label{Case4}
\scalebox{0.6}{
 \begin{tabular}{|c|c|c|c|c|c|c|c|c|c|c|c|c|c|c|c|}
\hline
 \multicolumn{4}{c|}{LINEAR} & \multicolumn{4}{c|}{CUBIC} & \multicolumn{4}{c|}{MPCSL} & \multicolumn{4}{c|}{TCHEB 3}  \\
\hline
 NbM &  Err & Rate  & Time  &  NbM &  Err & Rate  & Time  & NbM &  Err & Rate  & Time &  NbM &  Err & Rate  & Time  \\
\hline
200         &  0.0445   &        &  45      &  80     & 0.023   &        &  7    &  80    & 0.0234     &        &  8    & 20     & 0.0503   &    &   14    \\ 
400         &  0.0249   &        &  81     &  160    & 0.0143  &         &  13   &  160   & 0.0143     &         &  12    & 40     & 0.032   &   &  27    \\ 
800         &  0.014    &  0.83  &  157     &  320    & 0.00879  & 0.66  &  29   &  320   & 0.0088     & 0.72    & 28   & 80     &  0.020     & 0.61   &  46  \\
1600        &  0.0078   &  0.81  &  307     &  640    & 0.00529 & 0.66  &  55   &  640   &  0.0052    & 0.62     & 52    & 160     & 0.0125    &  0.68  & 95\\
3200        &  0.004    &  0.71  &  612     &  1280   & 0.00314 & 0.70  &  113   &  1280  &  0.0031    & 0.77    &  108 & 320     & 0.0075   &  0.59  &  180   \\
\hline
 \multicolumn{4}{c|}{LEGEND 2} & \multicolumn{4}{c|}{LEGEND 3} & \multicolumn{4}{c|}{BERN 2} & \multicolumn{4}{c|}{BERN 3} \\
\hline
 NbM &  Err & Rate  & Time  &  NbM &  Err & Rate  & Time  & NbM &  Err & Rate  & Time  &  NbM &  Err & Rate  & Time \\
\hline
 20       &  0.054   &            &   3      &  20      & 0.0520   &         &  7   &  100       & 0.0575    &         &    34    &  100       & 0.049   &          &   61 \\ 
 40       &  0.034   &            &   7      &  40      & 0.033    &         &  10  & 200       & 0.032     &          &    63    &   200      &  0.0281  &        &   112 \\ 
 80       &  0.021   &   0.62     &   13      & 80      & 0.020    & 0.55    &  18  &  400       & 0.0187    &   0.93     &    117   &   400      &  0.016  &  0.79     &   217 \\ 
 160      &  0.013   &   0.70     &   23      & 160     & 0.0127   & 0.85    &  34  &  800       & 0.0106    &   0.71     &    235   &   800      &  0.009  &  0.79     &   435  \\
 320      &  0.007   &   0.41     &   47      & 320     & 0.0077   & 0.55    &  70  &  1600      & 0.006     &   0.81     &    464   &   1600      & 0.0052   & 0.88   &  867  \\
\hline
\end{tabular}
}
 \end{table}

\subsubsection{A 2D dimensional control problem}
We use here the stochastic target problem from \cite{Bouchard}. 
Coefficients are given by:
\begin{eqnarray}
f_a(t,x) & =& 0,  c_a(t,x)  =  0,  b_a(t,x) = \left( \begin{array}{l}
                                                        - \frac{\kappa^2}{2} \\
                                                        - \frac{\mu}{\kappa} a
                                                        \end{array} \right), \nonumber \\ 
\sigma_a(t,x) & = & \left (  \begin{array}{l} 
                              \kappa  \\
                               a      
                             \end{array}
                       \right),  \mathop{A} = \{ a \in \R \}, Q =  (0,1] \times [-3,3] \times [0,1] \nonumber 
\end{eqnarray}
The initial condition for a European call with strike $K$ is given by:
\begin{eqnarray}
g(x_1,x_2) & =& x_2 (S_0 e^{x_1}-K)^{+}  \nonumber
\end{eqnarray}
We take  $\kappa = 0.4$, $\mu =1$, $S_0=K=1$. The value function $u$ is convex in $x_2$. Its Legendre-Fenchel transform  $u^*$ can be estimated by Monte Carlo method and we numerically calculate our reference solution  $u = u^{**}$ with
\begin{eqnarray}
u(t,x_1,x_2)& = & \max_{q} \left[ x_2 q - \mathbb{E} \left[ (q e^{-\frac{\mu^2}{2 \kappa^2} + \frac{\mu}{\kappa} g} - (x_1 e^{ -\frac{\kappa^2}{2} + \kappa g })^{+})^{+} \right] \right] \nonumber
\nonumber \end{eqnarray}
and $g \sim    \mathbf{N}(0,1)$.
Numerically we have to truncate the domain in $x_1$. The maximum control is truncated to $10$ and discretized with 2000 values.
The number of time steps is equal to 1600 So $h=6.25e-4$. We give the error on a sub domain of the domain of resolution $[-1.6,1.6]\times [0,1]$.
CPU times are given for 192 cores. Similarly to the previous case, the higher order scheme are not superior to the LINEAR scheme in term of rate of convergence
but for coarse meshes the higher order schemes are clearly superior.
\begin{table}
\caption{Test case 5  } \label{Case5}
\scalebox{0.6}{
 \begin{tabular}{|c|c|c|c|c|c|c|c|c|c|c|c|c|c|c|c|}
\hline
 \multicolumn{4}{c|}{LINEAR} & \multicolumn{4}{c|}{CUBIC} & \multicolumn{4}{c|}{MPCSL} &  \multicolumn{4}{c|}{TCHEB 3}\\
\hline
 NbM &  Err & Rate  & Time  &  NbM &  Err & Rate  & Time  & NbM &  Err & Rate  & Time & NbM &  Err & Rate  & Time  \\
\hline
  40       & 1.21   &               &   21     &  10   & 0.85   &            & 7   & 10   & 0.341   &         &  6        &   10   & 0.231  &    &       41   \\ 
  80       & 1.008  &               &   86     &  20   & 0.086  &           & 10  & 20   & 0.0746   &         &  10 &    20    & 0.077  &   &   127  \\ 
  160       & 0.619  &              &   285    &  40   & 0.045  & 4.23      & 33  & 40   & 0.118   &          &  33  &   40    & 0.224  &    &  394   \\ 
  320       &  0.319  &  0.36       &   1176   &  80   & 0.041  & 3.33      & 131 & 80   & 0.041   &          &  129 &   80   &  0.175  &   &  1577   \\ 
  640       &  0.192  &  1.24       &   4647   & 160   & 0.038  & 0.55      & 455 & 160  & 0.039   &  5.48    &  450 &   160      &  0.093    &   -0.72     &  6312 \\ 
\hline
 \multicolumn{4}{c|}{LEGEND 2} & \multicolumn{4}{c|}{LEGEND 3} & \multicolumn{4}{c|}{BERN 2} & \multicolumn{4}{c|}{BERN 3} \\
\hline
 NbM &  Err & Rate  & Time  &  NbM &  Err & Rate  & Time  &  NbM &  Err & Rate  & Time  &  NbM &  Err & Rate  & Time  \\
\hline
 10       &  0.032  &             &  6     &  10   &  0.035  &                & 30  &   40       &   0.417        &          & 266    &  40    & 0.395   &          &  916     \\ 
 20       &  0.017  &        &  21     & 20    &  0.019   &             & 87  &   80       &  0.2950  &         &  929    &  80    &  0.2647  &      &   3657  \\ 
 40       &   0.013 & 1.96   &  83     & 40    &  0.0155  &  2.12       & 270  &  160       & 0.1785  &    0.07     &  3711   &  160    & 0.1336  &   0.01 &  14623    \\ 
 80       &   0.014 &        &  293    & 80    &  0.0132  &  0.66       & 1077  &   320       & 0.069   &   0.08       & 14530   &     &         &            &   \\ 
 \hline
\end{tabular}
}
\end{table}
\section{Appendix}
\subsection{Some results on Lagrange interpolators} 
\label{Annexe1}
For a $d$ dimensional grid $X =X_{N}^d$, the interpolation operator is the composition of interpolators 
 $ I^{X}_{N}(f)(x) =  I^{X_{N},1}_{N} \times I^{X_{N},2}_{N} ... \times I^{X_{N},d}_{N} (f)(x)$ where  $I^{X_{N},i}_{N}$ is the one dimensional interpolator in direction $i$.\\
If we divide the domain  $I =[a_1,b_1] \times  .. \times [a_d,b_d] $ in meshes  $ \Delta x = (\Delta x_1,\Delta x_2..,\Delta x_d)$ such that  
$$ M_{i_1,..i_d} = [a_1 + i_1 \Delta x_1, a_1 + (i_1+1) \Delta x_1] \times .. \times  [a_d + i_d \Delta x_d, a_d + (i_d+1) \Delta x_d]$$
and if a Lagrange interpolation is used on each mesh  $M_{i_1,...,i_d}$  for the function  $g(x)=  f(a_1 + i_1 \Delta x_1 + (x_1+1) \frac{\Delta x_1}{2},..,a_d + i_d \Delta x_d + (x_d+1) \frac{\Delta x_d}{2} )$  then (see for example \cite{quarteroni} page 270)
for  $f \in C^{k+1}(I)$,
\begin{eqnarray}
|| f - I^X_{N,\Delta x} f ||_{\infty} \le C(N)  \sum_{i=1}^d  \Delta x^{k+1}_i \sup_{x \in [-1,1]^d} |\frac{\partial^{k+1} f}{\partial x_i^{k+1}}|
\label{InterpRegulier}
\end{eqnarray}
When $f$ is only $K$ Lipschitz, using Jackson's theorem we get :
\begin{eqnarray}
|| I^{X}_{N}(f) - f ||_\infty   \le  C K \sup_{i} { \Delta x_i} \sqrt{d} \frac{(1+\lambda_N(X))^d}{N+2}
\label{interpolX}
\end{eqnarray}

\subsection{Some results on Bernstein polynomials}
\label{Annexe3}
The approximation $B_N(f)$  of a function  $f : [0,1] \longrightarrow \R$ is the polynomial 
\begin{eqnarray}
B_N(f)(x) &= & \sum_{i=0}^N f(\frac{i}{N}) P_{N,i}(x)  \mbox{ where } P_{N,i}(x)  = \left( \begin{array}{l} N \\ i \end{array} \right) x^i (1-x)^{N-1}.\nonumber
\end{eqnarray}
It is important to notice that it is not an interpolation. Only points  $0$ and $1$ are interpolated.
By tensorization \cite{Gal}
\begin{eqnarray}
B_{N_1,..,N_d}(f)(x_1,..,x_d) =  \sum_{i_1=0}^{N_1} ... \sum_{i_d=0}^{N_d} \left[  \prod_{j=1}^d P_{N_j,i_j}(x_j) \right] f(\frac{i_1}{N_1},...,\frac{i_d}{N_d}) \nonumber
\end{eqnarray}
By introducing the modulus of continuity 
\begin{eqnarray}
w_1(f,\delta_1,..,\delta_d) & =&  \sup \left \{ |f(x_1,..,x_d) -f(y_1,..,y_d)| ; |x_i-y_i| \le \delta_i, i=1,..d \right\} \nonumber
\end{eqnarray}
we have the following estimation  \cite{Gal}
\begin{eqnarray}
|f(x_1,..,x_d) - B_{N_1,..,N_d}(f)(x_1,..,x_d) | & \le &  C w_1(f, \frac{1}{\sqrt{N_1}},...,\frac{1}{\sqrt{N_d}})            \nonumber
\end{eqnarray}
For a regular function the convergence rate is low 
\begin{eqnarray}
|f(x_1,..,x_d) - B_{N,..,N}(f)(x_1,..,x_d) | & \le & \frac{C}{N} \sum_{i}^d |\frac{\partial^2 f}{\partial x_i^2}(x_1,..,x_d)|.  \nonumber
\end{eqnarray}
The weights associated to this approximation are positive and independent on the function so this operator is monotone. It is known that it preserves
the convexity. Many other approximations with similar properties can be developed  \cite{Anasta}.

\begin{acknowledgements}
Special thank to Nadia Oudjane and St\'ephane Villeneuvre for their  careful reading. Special thanks to Romuald Elie for providing an analytical solution to
test case \reff{stochasticTarget1D}.
\end{acknowledgements}

\end{document}